\documentclass[12pt,oneside,reqno]{amsart}
\usepackage{txfonts}
\usepackage{bbm}
\usepackage{amsmath}
\usepackage{graphicx}
\usepackage{mathrsfs}
\usepackage{stmaryrd}
\usepackage{amsfonts}
\usepackage{enumerate,amsmath,amssymb,amsthm}

\numberwithin{equation}{section}
\newtheorem{theorem}{Theorem}[section]
\newtheorem{thm}[theorem]{Theorem}

\newtheorem{lem}[theorem]{Lemma}
\newtheorem{proposition}[theorem]{Proposition}
\newtheorem{prop}[theorem]{Proposition}
\newtheorem{corollary}[theorem]{Corollary}
\newtheorem{assumption}[theorem]{Assumption}
\theoremstyle{definition}

\newtheorem{defn}[theorem]{Definition}

\theoremstyle{remark}

\newtheorem{rem}[theorem]{Remark}
\numberwithin{equation}{section}

 \DeclareMathAlphabet{\mathpzc}{OT1}{pzc}{m}{it}

 \def\dif{{\mathord{{\rm d}}}}

 \newcommand{\E}{\mathbb{E}}            % expectation
 \newcommand{\e}{\varepsilon}

 \newcommand{\Ll}{\langle}
 \newcommand{\Rr}{\rangle}

 \newcommand{\N}{\mathbb{N}}
 \newcommand{\R}{\mathbb{R}}
 
 \newcommand{\PP}{\mathbb{P}}
 \newcommand{\mcl}{\mathcal}

 \newcommand{\Be}{\begin{equation}}
 \newcommand{\Ee}{\end{equation}}
 \newcommand{\Bs}{\begin{split}}
 \newcommand{\Es}{\end{split}}
  \newcommand{\Bes}{\begin{equation*}}
 \newcommand{\Ees}{\end{equation*}}
 \newcommand{\BT}{\begin{thm}}
 \newcommand{\ET}{\end{thm}}
 \newcommand{\Bp}{\begin{proof}}
 \newcommand{\Ep}{\end{proof}}
 \newcommand{\BL}{\begin{lem}}
 \newcommand{\EL}{\end{lem}}
 \newcommand{\BP}{\begin{proposition}}
 \newcommand{\EP}{\end{proposition}}
 \newcommand{\BC}{\begin{corollary}}
 \newcommand{\EC}{\end{corollary}}
 \newcommand{\BR}{\begin{rem}}
 \newcommand{\ER}{\end{rem}}
 \newcommand{\BD}{\begin{defn}}
 \newcommand{\ED}{\end{defn}}
 \newcommand{\BI}{\begin{itemize}}
 \newcommand{\EI}{\end{itemize}}
 \newcommand{\eqn}{equation}
 \newcommand{\tl}{\tilde}

 \newcommand{\bg}{\big}

 \newcommand{\re}{{\rm e}}

\begin{document}
\title
[Exponential mixing of 2D SDEs forced by degenerate L\'evy noises]{Exponential mixing of 2D SDEs forced by degenerate L\'evy noises}

\author[L. Xu]{Lihu Xu}
\address{Department of Mathematics, Brunel University,
Kingston Lane,
Uxbridge,
Middlesex UB8 3PH, United Kingdom}
\email{xulihu2007@gmail.com}

%\subjclass[2000]{}
%\keywords{}
%\date{}
%%% ----------------------------------------------------------------------
\maketitle
\begin{abstract} \label{abstract}
\noindent We modify the coupling method established in \cite{Shi08} and develop
a technique to prove the exponential mixing of a 2D stochastic system
forced by degenerate L\'evy noises. In particular, these L\'evy noises include
$\alpha$-stable noises ($0<\alpha<2$). Thanks to the stimulating discussion \cite{Ner11}, this technique is promising to study
the exponential mixing problem of SPDEs driven by degenerate \emph{symmetric} $\alpha$-stable noises.
%nonlinearity is bounded, then the system is ergodic and strong
%mixing. Then we show that the system is   exponentially mixing
%provided that the  nonlinearity, or its  Lipschitz constant, are
%sufficiently small. \\
%{\bf Keywords}: SDEs driven by degenerate $\alpha$-stable noises, coupling, exponential mixing. \\
%{\bf Mathematics Subject Classification (2000)}: \ {60H15, 60J75,  37A25}.
\end{abstract}
%%% ----------------------------------------------------------------------

\section{Introduction}
We shall study in this paper the exponential ergodicity of degenerate stochastic evolution
equation
\begin{\eqn} \label{e:MaiSPDE}
\begin{cases}
\dif X_1(t)=[-\lambda_1 X_1(t)+F_1(X(t))]\dif t+\dif z(t), \\
\dif X_2(t)=[-\lambda_2 X_2(t)+F_2(X(t))]\dif t
\end{cases}
\end{\eqn}
where $X(t)=(X_1(t), X_2(t))^T \in \R^2$ for every $t \ge 0$, $\lambda_2, \lambda_1>0$,
$F: \R^2 \rightarrow \R^2$ is bounded and Lipschitz,
$z(t)$ is a one dimensional L\'evy process satisfying Assumption \ref{a:} below.
We often simply write the above equation as the following form:
\Be \label{e:SimEqnFor}
\dif X(t)=[AX(t)+F(X(t))]\dif t+\dif Z_t,
\Ee
where $A=diag\{-\lambda_1, -\lambda_2\}$ and $Z_t=[z(t),0]^{T}$.

Since the end of the last century, the ergodicity of stochastic systems forced by degenerate noises
has also been intensively studied, see \cite{EMat01, Hai02, HM06, HaiMat11} for the SPDEs with degenerate Wiener noises and
\cite{KS-2001, KS-2002, KS10, Shi04, Shi06, Shi08, nersesyan-2008} for those forced by kick noises. However, there seems no ergodicity result for the stochastic systems driven by
degenerate L\'evy
jump noises. To our knowledge, this paper seems the first one in this direction.

The main novelty of the present paper is that we obtain
the exponential ergodicity for a family of 2D SDEs driven by
a large class of degenerate L\'evy jump noises which include
$\alpha$-stable noises with $0<\alpha<2$. In \cite{KS-2001, KS-2002, KS10, Shi04, Shi06, Shi08}, the authors assumed that
the kick noises come periodically and are bounded or with exponential moments.
\cite{nersesyan-2008} studied polynomial mixing for the complex Ginzburg-Landau equation
driven by random kick noises with all $p>0$ moments. Clearly,
all these assumptions in the above literatures rule out the interesting L\'evy noises only
with some $p>0$ moment such as $\alpha$-stable noises.

Let us also compare our ergodicity result with those known for
SDEs and SPDEs forced by L\'evy noises.
\cite{PSXZ11} established the exponential mixing for a family of SPDEs with a form similar to Eq. \eqref{e:SimEqnFor}
under total variational norm,
provided that the noises are \emph{non-degenerate} $\alpha$-stable with $1<\alpha<2$.
The non-degeneracy assumption and the regime of $\alpha \in (1,2)$ are
crucial for getting the strong
Feller property and applying the Lyapunov function technique.
The new point in the present paper is that our noises are degenerate and can be
$\alpha$-stable with $0<\alpha \le 1$. It is well known that $\alpha=1$ is a critical point
of $\alpha$-stable noises and $\alpha$-stable type operators. Many nice results in the case of $\alpha>1$ can not be extended to the case $\alpha \in (0,1]$ (\cite{Pr12}).
 \cite{Ku09} established some nice criteria of the exponential mixing
(under total variation norm) for a family of finite dimensional SDEs driven by jump noises which
include some one dimensional equations driven by $\alpha$-stable noises.

Our approach is by
modifying the coupling method established in \cite{Shi08, Shi04}, which
has been applied to study the ergodicity problems of
many degenerate stochastic systems (\cite{KS-2001, KS-2002, KS10, Shi04, Shi06, Shi08, nersesyan-2008}). 
Roughly speaking, we follow the idea in \cite{Shi08, Shi04} to split the dynamics into two parts, one with noises and the other 
with strong dissipation. We apply a maximal coupling for the part with noises to mix the coupling chain and take advantage of strong  
dissipation to control the part without noises. There are two different points between our modification and the methods in \cite{Shi04, Shi08}. 
We sample the solution Markov chain according to the moment that a jump larger than $K$ comes (see Section \ref{s:ConCou}), 
while the Markov chains in \cite{Shi04} are sampled periodically thanks to their special periodical kick noises. To handle this new 
random effect, we need to estimate some more complicated stopping times in sequel. The other point is that we use the jumps larger than 
$K$ to construct the coupling chains, while \cite{Shi08, Shi04} take the advantage of small jumps.

It is natural to ask
whether our exponential ergodicity result can be extended to SPDEs forced by finite dimensional L\'evy noises such as $z(t)=\sum_{k=1}^n z_k(t)$ with $z_1(t),...,z_n(t)$ being a sequence of independent purely jump noises. Unfortunately, it seems our technique is not applicable even for the case of 3d SDEs driven by 2d L\'evy jump noises. Let us point out the difficulty (very) roughly by the following models.
Consider
\begin{\eqn} \label{e:Toy}
\begin{cases}
\dif X_1(t)=[-\lambda_1 X_1(t)+F_1(X(t))]\dif t+\dif z_1(t), \\
\dif X_2(t)=[-\lambda_2 X_2(t)+F_2(X(t))]\dif t+\dif z_2(t), \\
\dif X_3(t)=[-\lambda_3 X_3(t)+F_3(X(t))]\dif t
\end{cases}
\end{\eqn}
where $\lambda_1, \lambda_2, \lambda_3>0$, $F: \R^3\rightarrow \R^3$ is bounded and Lipschitz, $z_1(t)$ and $z_2(t)$ are
independent 1d L\'evy jump processes. Let $\lambda_3$ be sufficiently large to make
the dissipative term $-\lambda_3 X_3(t)$
dominate the third equation. For the first two equations, when $z_1(t)$ has a jump $\eta_1$ at some moment $\tau$, there are no jumps for $z_2(t)$ at $\tau$ almost surely. We can take the advantage of the jump $\eta_1$ to control the growth of the sample paths of $X_1(t)$ in a short time interval $[\tau, \tau+\delta)$.
Due to the lack of the random effect, the growth of the sample paths of $X_2(t)$
can not be handled in $[\tau, \tau+\delta)$.

From the stimulating discussion \cite{Ner11}, our technique is promising to study
the exponential mixing problem of
SPDEs driven by finite dimensional \emph{symmetric} $\alpha$-stable processes.
These type of processes have a nice representation by $W_{S_t}$ with $W_t$ being a standard
$n$-dimensional Brownian motion and $S_t$ being an $\alpha/2$-stable subordinator. When a jump of $S_t$ comes,
all the $n$ directions of $W_{S_t}$ jumps simultaneously, thus the difficulties in Eq. \eqref{e:Toy} will not appear any more. Symmetric $\alpha$-stable processes have recently studied by both analysis and probability communities (\cite{ChKiSo12,Wa,Zh}).

The structure of the paper is as follows. Section 2 introduces some notations and gives our main theorem. A coupling Markov chain is introduced in Section 3. Section 4 introduces some stopping times related to this Markov chain, which is used to prove the main theorem in Section 5.
\vskip 2mm

{\bf Acknowledgements:} The author gratefully thanks Jerzy Zabczyk for the stimulating discussions and many useful suggestions. He also gratefully thanks
Armen Shirikyan for patiently teaching him the coupling method in the paper \cite{Shi08}. Special thanks are due to Vahagn
Nersesyan for numerous useful suggestions, carefully reading the paper and stimulating discussions about
studying the exponential mixing problem of
SPDEs driven by \emph{symmetric} $\alpha$-stable noises.

\section{Notations and main results}
Denote by $B_b(\R^2)$ the Banach space of bounded Borel-measurable functions $f:\R^2 \rightarrow \R$ with the norm
$$
\|f\|_{0}:=\sup_{x \in \R^2} |f(x)|.
$$
Denote by $L_b(\R^2)$ the Banach space of global Lipschitz bounded functions $f: \R^2 \rightarrow \R$ with the norm
$$
\|f\|_{1}:=\|f\|_0+\|f\|_{{\rm Lip}}.
$$
where
$\|f\|_{{\rm Lip}}:=\sup_{x \neq y}\frac{|f(x)-f(y)|}{|x-y|}.$
\smallskip
%Let $\mcl B(\R^2)$ be the Borel $\sigma$-algebra on $\R^2$ and let~$\mcl P(\R^2)$ be
%the set of probabilities on $(\R^2,\mcl B(\R^2))$. Recall that the total variation distance
%between two measures $\mu_1, \mu_2 \in \mcl P(\R^2)$ is defined by
%$$
%\|\mu_1-\mu_2\|_{\rm TV}
%=\frac 12 \sup_{\stackrel{f \in B_b(\R^2)}{\|f\|_0=1}} |\mu_1(f)-\mu_2(f)|
%=\sup_{\Gamma\in\mcl B(\R^2)}|\mu_1(\Gamma)-\mu_2(\Gamma)|.
%$$
%Given a random variable
%$X$, we shall use $\mcl L(X)$ to denote the distribution of $X$.
%\vskip 3mm

\subsection{Some preliminary of L\'evy process (\cite{Ber96})}
Let $(z(t))_{0 \le t<\infty}$ be a one-dimensional purely jumping L\'evy process with
the characteristic function
\Bes
\E {\rm e}^{{\rm i} \xi z(t)}={\rm e}^{t \psi(\xi)}, \ \ \ t \ge 0,
\Ees
where $\psi(\xi)$ is the symbol of $z(t)$. Recall
\Bes
\psi(\xi)=\int_{\R \setminus \{0\}} \left({\rm e}^{{\rm i} \xi y}-1-{\rm i} \xi y 1_{\{|y| \le 1\}}\right) \nu(\dif y),
\Ees
where $\nu$ is the L\'evy measure associated with $z(t)$.

For every $t>0$,
a Poisson random measure $N(t,.)$ is defined by
$$
N(t,\Gamma):=\sum_{s\in(0,t]}1_{\Gamma}\big(\Delta z(s)\big), \ \ \ \ \forall \ t>0, \ \forall \ \Gamma\in \mcl B(\R \setminus\{0\}),
$$
where $\Delta z(s)=z(s)-z(s-)$. For every $K>0$, define
$$z^K(t):=\sum_{0 \le s \le t} \Delta z(s) 1_{\{\Delta z(s) \ge K\}}.$$
Further define
$$\Gamma_K:=(-\infty,-K] \cup [K,\infty), \ \ \gamma_K:=\nu(\Gamma_K),$$
$\gamma_K$ is a decreasing function of $K$ and $\gamma_K<\infty$ for $K>0$.

Let $\tl \tau_1, \tl \tau_2, \dots, \tl \tau_n, \dots$
be a sequence of random times such that
$$\tl \tau_1, \tl \tau_2-\tl \tau_1,
\dots, \tl \tau_n-\tl \tau_{n-1},\cdots$$
are independent exponential random variables
with parameter $\gamma_K$, i.e.
$$\PP(\tl \tau_n-\tl \tau_{n-1}>s)=\re^{-\gamma_K s}, \ \ \ s>0.$$
It is well known that $z^K(t)$ can also be represented by
\Be \label{e:Z>e}
z^{K}(t)=\sum_{k \ge 1} \eta_k 1_{\{\tl \tau_k \le t\}}
\Ee
where $\eta_k$ are independent random variable sequences with distribution
\Be \label{e:NuK}
\nu_K:=\frac{1}{\gamma_K} \nu \big |_{\Gamma_K}.
\Ee
\vskip 2mm

\begin{assumption}
\label{a:}
Assume that the following conditions hold:
\begin{itemize}
\item [(A1)]  For every $\lambda>0$ and
$p \in (0, \alpha)$ with $\alpha \in (0,2)$,
$$\sup_{0 \le t<\infty} \E \left|\int_0^t \re^{-\lambda (t-s)} \dif z(s)\right|^p<\infty.$$
\item [(A2)] For some $K>0$, $\nu_K$ has a density function $p_K$ such that for all $z_1,z_2 \in \R$
\Bes
\int_\R |p_K(z-z_1)-p_K(z-z_2)| \dif z \le  \beta_1|z_1-z_2|^{\beta_2},
\Ees
where $\beta_1, \beta_2>0$ are constants only depending on $K$.
\item [(A3)] There exist some $M>0$ and some $\beta_0=\beta_0(K,M) \in (0,2)$ such that if $|z_1|+|z_2| \le M$,
\Bes
\int_\R |p_K(z-z_1)-p_K(z-z_2)| \dif z \le \beta_0.
\Ees
\item [(A4)] $\gamma_{K} \ge 2 \beta_2 \|F\|_{{\rm Lip}}$.
\end{itemize}
\end{assumption}

\begin{rem}
The number '$2$' in '$\gamma_{K} \ge 2 \beta_2 \|F\|_{Lip}$' of (A4) can be replaced by any number $c>1$.
We choose the special '$2$' to make the computation in sequel more simple. Roughly speaking, (A4) means that
the process $(z(t))_{t \ge 0}$ has sufficiently many jumps bigger than $K$. The number $M$ will be chosen in Theorem
\ref{l:TauPro}.
\end{rem}

\begin{proposition}
An $\alpha$-stable process $(z(t))_{t \ge 0}$ with $0<\alpha<2$
satisfies Assumption \ref{a:}.
\end{proposition}

\begin{proof}
Recall that the L\'evy measure of the $\alpha$-stable process has the form
$$\nu(\dif x)=\frac{c_\alpha}{|x|^{\alpha+1}} 1_{\{|x|>0\}}\dif x,$$
where $c_\alpha$ is some not important constant.
It is easy to see that $\gamma_K \uparrow \infty$ as $K \downarrow 0$, thus (A4) holds. Since $z(t)$ has the characteristic function ${\rm e}^{-|\xi|^\alpha t}$, it is easy to check that $\int_0^t e^{-\lambda(t-s)} dz_s$ has characteristic function
$\exp \big\{-\frac{1-e^{-\alpha \lambda t}}{\alpha\lambda}|\xi|^\alpha\big\}$.
This, together with (3.2) of \cite{PZ09}, immediately gives (A1).

For every $K>0$, we have
$$p_K(z-z_i)=\frac{\alpha  K^\alpha}{2} \frac{1}{|z-z_i|^{\alpha+1}} 1_{\{|z-z_i|>K\}}, \ \ \ \ (i=1,2).$$
Since the supports of the functions $p_K(z-z_1)$ and $p_K(z-z_2)$ have overlaps, it holds that
\Bes
\begin{split}
\int_{\R} |p_K(z-z_1)-p_K(z-z_2)| \dif z<\int_{\R} p_K(z-z_1)\dif z+\int_{\R}p_K(z-z_2) \dif z=2.
\end{split}
\Ees
It is easy to check that for all $M>0$, there exists some $\beta_0 \in (0,2)$ depending 
on $M$ and $K$ such that (A3) holds.

It remains to verify (A2). By the easy fact $1-\left(\frac{1}{1+r}\right)^{\alpha} \le (\alpha+1)r$
for $|r|<\frac{1}{2+2\alpha}$,
when $|z_1-z_2| \le \frac{K}{2\alpha+2}$,
\Bes \label{e:A3-1}
\int_{\R} |p_K(z-z_1)-p_K(z-z_2)| \dif z \le \frac{2\alpha+2}K |z_2-z_1|.
\Ees
As $|z_2-z_1|>\frac{K}{2\alpha+2}$, we have $\frac{4\alpha+4}{K} |z_2-z_1|>2$ and thus
\Bes \label{e:A3-1}
\int_{\R} |p_K(z-z_1)-p_K(z-z_2)| \dif z \le \frac{4\alpha+4}{K} |z_2-z_1|.
\Ees
Take $\beta_1=\frac{4\alpha+4}{K}$ and $\beta_2=1$, we immediately get (A2).
\end{proof}
\vskip 3mm

\subsection{Main result} Let us first show Eq. \eqref{e:MaiSPDE} is well-posed and then
give the main theorem.
\begin{thm} \label{t:SolEU}
For any $x \in \R^2$, problem \eqref{e:MaiSPDE} has a unique strong solution
$(X^x(t))_{t \ge 0}$ with the form:
\begin{\eqn} \label{e:MilSol}
X^x(t)={\rm e}^{At}x+\int_0^t {\rm e}^{A(t-s)} F(X^x(s))\dif s+\int_0^t {\rm e}^{A(t-s)}
\dif Z_s.
\end{\eqn}
Moreover,  $(X^x(t))_{t \ge 0}$ has a c$\grave{a}$dl$\grave{a}$g version in $\R^2$ and is an $\R^2$-valued Markov process starting from $x$.
\end{thm}

\begin{proof}
The existence, uniqueness and Markov property of the strong solution have been proved in
\cite{PZ09}. Since $Z_t$ clearly has a c$\grave{a}$dl$\grave a$g version,
$\int_0^t e^{A(t-s)} dZ_s$ also has a c$\grave a$dl$\grave a$g one.
The other two terms on the r.h.s. of \eqref{e:MilSol} are both
continuous, so $(X^x(t))_{t \ge 0}$ is c$\grave a$dl$\grave a$g.
\end{proof}
\vskip 2mm

Denote by $(P_t)_{t\geq 0}$ the Markov semigroup associated with~\eqref{e:MaiSPDE}, i.e.
\begin{\eqn*}
P_t f(x):=\E\left[f(X^x(t))\right], \quad f \in B_b(\R^2),
\end{\eqn*}
and by $(P^{*}_t)_{t \geq 0}$ the dual semigroup acting on $\mcl P(\R^2)$.
Our \emph{main result} is the following ergodic theorem which will be proven in the last section.

\begin{thm}\label{t:MaiThm}
Let $\lambda_1>0$ and Assumption \ref{a:} both hold. There exists some $\lambda_0=\lambda_0(\|F\|_{1}, M, \beta_0, \beta_1, \beta_2)$, where $M, \beta_0, \beta_1, \beta_2$ are as in Assumption \ref{a:}, such that as $\lambda_2>\lambda_0$, the system \eqref{e:MaiSPDE} is exponentially ergodic under the weak topology of
$\mcl P(\R^2)$. More precisely, there exists a probability measure
$\mu \in \mcl P(\R^2)$ so that for any $p \in (0, \alpha)$ and
any measure $\tl \mu \in \mcl P(\R^2)$ with finite $p^{\rm th}$ moment, we have
\begin{equation} \label{2.5}
|\Ll P^{*}_t \tl \mu,f\Rr-\Ll \mu,f\Rr| \leq C e^{-ct}\|f\|_1\bigg(1+\int_{\R^2} |x|^{p}
\tl \mu(\dif x)\bigg), \ \ \ \ \forall \ f \in L_b(\R^2),
\end{equation}
where $C,c$ depend on $p, K, \|F\|_1, \beta_0, \beta_1, \beta_2, \lambda_1, \lambda_2, M$.
\end{thm}

\vskip 1mm

Let us briefly give the strategy of the coupling method we shall use (it is a modification of
the method established in \cite{Shi08}):

(i) Take a \emph{waiting} time $T$ (a fixed number) and define $\tau_0=0$, we look for the first jump after the time $T$ and record its moment by $\tau_1$. Similarly, we do not look for the next jump immediately after $\tau_1$ but do it after the time $\tau_1+T$, and so on. In this way we get a sequence of stopping time $\{\tau_k\}_{k \ge 0}$. The waiting time $T$ will play an important role for estimating the stopping times associated with the coupling Markov chain below.

 (ii) For any $x,y \in \R^2$, take two copies of processes $(X^x(t))_{t \ge 0}$ and $(X^y(t))_{t \ge 0}$,
consider the corresponding embedded Markov chains $(X^x(\tau_k))_{k \ge 0}$ and $(X^y(\tau_k))_{k \ge 0}$.
Using maximal coupling, we construct a coupling Markov chain $(S^{x,y}(k))_{k \ge 0}$ with
$S^{x,y}(k)=(S^x(k), S^y(k))$ for every $k \ge 0$. $(S^x(k))_{k \ge 0}, (S^y(k))_{k \ge 0}$
have the same distributions as those of $(X^x(\tau_k))_{k \ge 0}$ and $(X^y(\tau_k))_{k \ge 0}$ respectively.

 (iii) Define
\Bes
\begin{split}
& \tl \sigma=\inf\{k>0; |S^x(k)|+|S^y(k)| \le M\}, \\
& \hat \sigma=\inf\left\{k>0; |S^x(k)-S^y(k)| \ge \tau_k/\lambda^k_2\right\}.
\end{split}
\Ees
The exact $\hat \sigma$ is defined in \eqref{d:SigXY}, but the above simple version
captures the essential part of \eqref{d:SigXY}.
The main ingredient for showing Theorem \ref{t:MaiThm} is
$$\E[{\rm e}^{c\tl \sigma}]<\infty,   \ \ \ \ \ \PP(\hat \sigma=\infty|S_{\tl \sigma})>0.$$
The first inequality implies that the system $(S(k))_{k \ge 0}$ enters the $M$-radius ball
exponentially frequently. The second inequality means that as long as $(S(k))_{k \ge 0}$ is in that ball, there exists a set of sample paths
with positive probability such that $|S^x(k)-S^y(k)|$ converges to zero exponentially fast
as long as $\lambda_2$ is sufficiently large.
\vskip 3mm

Without loss of generality, from now on we assume
$$\lambda_1 \le \lambda_2.$$
Our method of course covers the case $\lambda_1>\lambda_2$, in which the dissipative term $A X(t)$ dominates the system and the exponential mixing can be shown by a quite easy argument (\cite{PXZ10}).
\vskip 1mm

\subsection{Some easy estimates about the solution}
\begin{lem} \label{l:SolEst}
 For every $x,y \in \R^2$ and $p \in (0, \alpha)$, if $\lambda_2>\|F\|_{{\rm Lip}}$ we have
\Bes
\E |X^x(t)|^p \le (3^{p-1} \vee 1)  \re^{-\lambda_1 pt} |x|^p+C, \ \ \ \ \forall \ t \ge 0,
\Ees
\Bes
\E |X^x(t)-X^y(t)|^p \le (3^{p-1} \vee 1)  \re^{-\lambda_1 pt} |x-y|^p+C, \ \ \ \ \forall \ t \ge 0,
\Ees
\Bes \label{e:X1X2Bou}
|X^x(t)-X^y(t)|\le \re^{t \|F\|_{{\rm Lip}}} |x-y|, \ \ \ \forall \ t \ge 0,
\Ees
\Bes \label{e:X1X2HBou}
|X^x_2(t)-X^y_2(t)| \le \bigg(\re^{-\lambda_2 t}+\frac{\|F\|_{{\rm Lip}}}{\|F\|_{{\rm Lip}}+\lambda_2} \ \re^{t \|F\|_{{\rm Lip}}}\bigg) |x-y|, \ \ \ \forall \ t \ge 0.
\Ees
where $a \vee b:=\max\{a,b\}$ for $a,b \in \R$ and $C$ depends on $p,\lambda,\|F\|_0$.
\end{lem}

\begin{proof}
By \eqref{e:MilSol} we have
\Bes
\begin{split}
\left|X^x(t)\right| & \le e^{-\lambda_1 t}\left|x\right|+\int_0^t e^{-\lambda_1 (t-s)}|F(X^x(s))| \dif s+\left|\int_0^t e^{-\lambda_1(t-s)} \dif z(s)\right|,
\end{split}
\Ees
this, together with (A1) and the assumption of $F$, immediately gives the first inequality. Observe
\Bes
\begin{split}
\left|X^x(t)-X^y(t)\right| \le e^{-\lambda_1 t} |x-y|+2\int_0^t e^{-\lambda_1(t-s)} \dif s \|F\|_0,
\end{split}
\Ees
from which the second inequality follows immediately.
Further observe
\Bes
|X^x(t)-X^y(t)| \le |x-y|+\int_0^t  \|F\|_{{\rm Lip}} |X^x(s)-X^y(s)| \dif s,
\Ees
from this we immediately get the third inequality by Gronwall's inequality.

We also easily have
\Bes
\begin{split}
|X^x_2(t)-X^y_2(t)|& \le e^{-\lambda_2 t} |x-y|+\int_0^t e^{-\lambda_2(t-s)} \|F\|_{{\rm Lip}} |X^x(s))-X^y(s)| \dif s,
\end{split}
\Ees
which, together with the third inequality, yields the fourth one.
\end{proof}

\section{Construction of the coupling Markov chain} \label{s:ConCou}
In this section, we construct a coupling Markov chain which will be used to prove the ergodicity result.
Let
\Be \label{e:TDef}
T>0 \ {\rm be \ a \ fixed \ number}
\Ee
to be determined later in Theorem \ref{l:TauPro}. We call $T$ the \emph{waiting} time, which means
that when a jump comes we do not look for the next jump immediately but do it after waiting
for a time $T$. This waiting time $T$ will play an important role in estimating the stopping times below.

 Define
\Be \label{e:TauDef}
\tau:=\inf\left\{t>T: |\Delta z(t)| \ge K\right\},
\Ee
$\tau$ is a stopping time with probability density
\Be \label{e:TauDen}
\gamma_K \exp\left\{-\gamma_K(t-T)\right\} 1_{\{t>T\}}.
\Ee
 Define $\tau_0:=0$ and
\begin{align*}
\tau_k:=\inf\left\{t>\tau_{k-1}+T: |\Delta z(t)| \ge K \right\} \ \ \
{\rm for \ all \ } k \ge 1.
\end{align*}
It is easy to see that $\{\tau_k\}_{k \ge 0}$ are a sequence of stopping times such that
\Be \label{e:TauKInd}
\{\tau_{k}-\tau_{k-1}\}_{k \ge 1} {\rm \ are \ independent \ and \ have \ the \ same \ density \ as}
\ \tau.
\Ee
\vskip 2mm

%Let us consider the following two 2D SDEs
%\Be
%d X(t)=[AX(t)+F(X(t))]dt+dZ_t, \ \ \ X(0)=x;
%\Ee
%\Be
%d Y(t)=[AY(t)+F(Y(t))]dt+dZ_t, \ \ \  Y(0)=y.
%\Ee

Since the solution of problem \eqref{e:MaiSPDE} with the initial data
$X(0)=x$ has a c$\grave{a}$dl$\grave{a}$g version,
$X^x(\tau_1-)$ is well defined with the form:
\Be
%\begin{split}
 X^x(\tau_1-)=e^{A \tau_1} x+\int_0^{\tau_1} e^{A(\tau_1-s)} F(X^x(s)) \dif s+\int_0^{\tau_1-} e^{A(\tau_1-s)} \dif Z_s,
%\\
%& Y(\tau_1-)=e^{A \tau_1} y+\int_0^{\tau_1} e^{A(\tau_1-s)} F(Y(s)) ds+\int_0^{\tau_1} e^{A(\tau_1-s)} dZ_s,
%\end{split}
\Ee
By \eqref{e:Z>e} and strong Markov property of $z(t)$, at the time $\tau_1$,
there is only one jump $\eta$ almost
surely and $\eta$ has the probability density $\nu_K$ (see \eqref{e:NuK}). Therefore,
\Bes
X^x(\tau_1)=X^x(\tau_1-)+\eta [1,0]^{T} \ \ \ a.s..
\Ees

%$\{S_x(k)\}_{k \ge 0}$ and $\{S_y(k)\}_{k \ge 0}$ are $\R^2$-valued Markov chain starting from
%$x$ and $y$ respectively. Of course, $S(k):=(S_x(k),S_y(k))$ is
%an $\R^2 \times \R^2$-valued Markov chain starting from $(x,y)$.

\noindent Denote by $P^{(1)}_{x}(.):\mcl B(\R^2) \rightarrow [0,1]$
the distribution of
$X^x(\tau_1-)$ for every $x \in \R^2$, and by $P^{(2)}_{\hat x}(.): \mcl B(\R^2) \rightarrow [0,1]$
the distribution of $\hat x+\eta [1,0]^T$ for every $\hat x\in \R^2$.
For any $A\in \mcl B(\R^2)$, define
\Be
P_x(A):=\int_{\R^2} P^{(2)}_{\hat x}(A) P^{(1)}_{x}(\dif \hat x),
\Ee
$(X^x(\tau_k))_{k \ge 0}$ is an $\R^2$-valued Markov chain with
transition probability $\left(P_{x}(.)\right)_{x \in \R^2}$.
\vskip 2mm

%Consider the two processes
%$(X^x(t))$ and $(X^y(t))$ starting from $x$ and $y$ respectively and the
%corresponding Markov chains
%$(X^x(\tau_k))_{k \ge 0}$ and $(X^y(\tau_k))_{k \ge 0}$.

For any random variable $X$, 
$${\rm denote \ by \ } \mcl L(X) \ {\rm the \ law \ of \ } X. $$
Let $\left(\xi_x(\hat x_1, \hat y_1), \xi_y(\hat x_1, \hat y_1)\right)$ be the maximal coupling of
$\mcl L(\hat x_1+\eta)$ and $\mcl L(\hat y_1+\eta)$, we have
\begin{lem} \label{l:TV1}
For every $\hat x_1, \hat y_1 \in \R$, we have
\Bes
\PP \left(\xi_x(\hat x_1, \hat y_1) \neq  \xi_y(\hat x_1, \hat y_1)\right) \le  \beta_1 |\hat x_1-\hat y_1|^{\beta_2}/2
\Ees
where $\beta_1, \beta_2$ are the constants in Assumption \ref{a:}. Furthermore, if $|\hat x_1|+|\hat y_1| \le M$, then
\Bes
\PP \left(\xi_x(\hat x_1, \hat y_1) \neq  \xi_y(\hat x_1, \hat y_1)\right) \le {\beta_0}/2.
\Ees
where $\beta_0$ is the constant in Assumption \ref{a:}.
\end{lem}

\begin{proof}
Since $(\xi_x(\hat x_1, \hat y_1),\xi_y(\hat x_1, \hat y_1))$ is the maximal coupling of
$\mcl L(\hat x_1+\eta)$ and $\mcl L(\hat y_1+\eta)$,
\Bes
\PP \left(\xi_x(\hat x_1, \hat y_1) \neq  \xi_y(\hat x_1, \hat y_1)\right)=\|\mcl L(\hat x_1+\eta)-\mcl L(\hat y_1+\eta)\|_{TV}.
\Ees
Note that the distributions  $\mcl L(\hat x_1+\eta)$ and $\mcl L(\hat y_1+\eta)$
have the densities $p_K(z-\hat x_1)$ and $p_K(z-\hat y_1)$ respectively, where $p_K$
is defined in Assumption \ref{a:}.
It is easy to see that
\Bes
\|\mcl L(\hat x_1+\eta)-\mcl L(\hat y_1+\eta)\|_{TV} \le \frac 12 \int_\R |p_K (z-\hat x_1)-p_K (z-\hat y_1)| \dif z,
\Ees
this, together with (A2) and (A3) of Assumption \ref{a:}, immediately implies the desired inequalities.
\end{proof}

For every $\hat x, \hat y \in \R^2$, define
\Be \label{e:BarXi}
\bar \xi_x(\hat x, \hat y):=\left[\begin{array}{lll} &\xi_x(\hat x_1, \hat y_1) \\
& \ \ \ \hat x_2
\end{array}
\right], \ \ \ \bar \xi_y(\hat x, \hat y):=\left[\begin{array}{lll} & \xi_y(\hat x_1, \hat y_1) \\
& \ \ \ \hat y_2
\end{array}
\right].
\Ee
Since $\mcl L(\xi_x(\hat x_1, \hat y_1))=\mcl L(\hat x_1+\eta)$ and $\mcl L(\xi_y(\hat x_1, \hat y_1))=\mcl L(\hat y_1+\eta)$,
we have
\Be  \label{e:LawBarXEta}
\begin{split}
\mcl L(\bar \xi_x(\hat x, \hat y))=\mcl L(\hat x+\eta [1,0]^{T}),  \ \ \  \mcl L(\bar \xi_y(\hat x, \hat y))=\mcl L(\hat y+\eta [1,0]^{T}).
\end{split}
\Ee

\noindent Denote the probability of $(X^x(\tau_1-),X^y(\tau_1-))$ with $(x,y) \in \R^2 \times \R^2$ by
\Be \label{e:P(1)}
P^{(1)}_{(x,y)}(.):\mcl B(\R^2 \times \R^2) \rightarrow [0,1],
\Ee
further denote the probability of
$(\bar \xi_x(\hat x, \hat y), \bar \xi_y(\hat x, \hat y))$ with $(\hat x,\hat y) \in \R^2 \times \R^2$ by
\Be \label{e:P(2)}
P^{(2)}_{(\hat x, \hat y)}(.):\mcl B(\R^2 \times \R^2) \rightarrow [0,1].
\Ee
For any $A \in \mcl B(\R^2 \times \R^2)$,
define
\Be \label{e:TraChnPro}
P_{(x,y)}(A):=\int_{\R^2 \times \R^2} P^{(2)}_{(\hat x,\hat y)} (A) P^{(1)}_{(x,y)}(\dif \hat x, \dif \hat y).
\Ee

\begin{prop} \label{p:CouChn}
There exists a probability space $(\tl \Omega, \tl {\mcl F}, \tl {\PP})$ and an $\R^2 \times \R^2$-valued Markov chain $\{S(k)\}_{k \ge 0}$ on $(\tl \Omega, \tl {\mcl F}, \tl {\PP})$ with transition probability
family $(P_{(x,y)})_{(x, y) \in \R^2 \times \R^2}$.
Moreover, for every $(x,y) \in \R^2 \times \R^2$,
the marginal chain $\{S^x(k)\}_{k \ge 0}$ has
the same distribution as $\{X^x(\tau_k)\}_{k \ge 0}$ and the marginal chain $\{S^y(k)\}_{k \ge 0}$  has
the same distribution as $\{X^y(\tau_k)\}_{k \ge 0}$.
\end{prop}

\begin{proof}
The construction of the coupling Markov chain is classical since the transition probability
family $(P_{(x,y)})_{(x, y) \in \R^2 \times \R^2}$ is ready.
To prove the other claim in the proposition, it suffices to
show that for all $x \in \R^2$, $y \in \R^2$, $A \in \mcl B(\R^2)$,
we have
\Be \label{e:CouSinEqn}
P_{(x,y)}(A \times \R^2)=P_{x}(A), \ \ P_{(x,y)}(\R^2 \times A)=P_{y}(A)
\Ee
where $(P_{x})_{x \in \R^2}$ is the
transition probability family of $(X(\tau_k))_{k \ge 0}$.

We only show the first equality of \eqref{e:CouSinEqn} since the other one can be proven similarly.
Recall that $P^{(1)}_{x}(.)$ is
the distribution of
$X^x(\tau_1-)$ and that $P^{(2)}_{\hat x}(.)$ is
the distribution of $\hat x+\eta [1,0]^T$.
It is clear that
$$P^{(1)}_{(x,y)}(. \times \R^2)=P^{(1)}_{x}(.), \ \ \ P^{(2)}_{(\hat x, \hat y)}(. \times \R^2)=P^{(2)}_{\hat x}(.),$$
where $P^{(1)}_{(x,y)}$ and $P^{(2)}_{(\hat x, \hat y)}$ are defined by \eqref{e:P(1)} and
\eqref{e:P(2)} respectively.
It follows from the definitions of $P_{(x,y)}(.)$ and $P_{x}(.)$ that
\Bes %\label{e:TraChnPro}
\begin{split}
P_{(x,y)}(A \times \R^2)&=\int_{\R^2 \times \R^2} P^{(2)}_{(\hat x,\hat y)}(A \times \R^2) P^{(1)}_{(x,y)} (\dif \hat x, \dif \hat y) \\
&=\int_{\R^2 \times \R^2} P^{(2)}_{\hat x}(A) P^{(1)}_{(x,y)} (\dif \hat x, \dif \hat y)=\int_{\R^2} P^{(2)}_{\hat x}(A) P^{(1)}_{x} (\dif \hat x)=P_{x}(A).
\end{split}
\Ees
\end{proof}

\section{Some estimates of the coupling chain $(S^{x,y}(k))_{k \ge 0}$}
$(S(k))_{k \ge 0}$ constructed in previous section is a Markov chain on the probability space
$(\tl \Omega, \tl {\mcl F}, \tl {\PP})$.
$(\tl \Omega, \tl {\mcl F}, \tl {\PP})$ is not necessarily
the same as $(\Omega, {\mcl F}, {\PP})$ on which $(X(t))_{t \ge 0}$ is located.
Without loss of generality, we assume that
\Be \label{a:SpaAss}
(\Omega, {\mcl F}, {\PP})=(\tl \Omega, \tl {\mcl F}, \tl {\PP}).
\Ee
Otherwise we can introduce the product space $(\tl \Omega \times \Omega, \tl {\mcl F} \times \mcl F, \tl {\PP} \times \PP)$
and consider $(S(k))_{k \ge 0}$ and $(X(t))_{t \ge 0}$ both on this new space. However,
this will make the notations unnecessarily complicated.
So, we always assume \eqref{a:SpaAss} and consider $(S(k))_{k \ge 0}$ and $(X(t))_{t \ge 0}$
on $(\Omega, \mcl F, \PP)$ from now on. 

For any $(x,y) \in \R^2 \times \R^2$, we denote 
$\{S^{x,y}(k)\}_{k \ge 0}$ the coupling Markov chain with initial state $(x,y)$. 
Recall that $\{S^{x}(k)\}_{k \ge 0}$ and $\{S^{y}(k)\}_{k \ge 0}$ denote the two marginal Markov chains. 
\vskip 3mm

Let $M>0$ and $d>0$ both be some number to be determined later, define the stopping times
\Be \label{d:SigM}
\tl \sigma(x,y,M):=\inf \left\{k>0; |S^x(k)|+|S^y(k)| \le M\right\},
\Ee
\Be \label{d:SigD}
\sigma(x,y,d):=\inf \left\{k>0; |S^x(k)-S^y(k)| \le d\right\},
\Ee
we write $\tl \sigma=\tl \sigma(x,y,M)$, $\sigma=\sigma(x,y,d)$ in shorthand if no confusions arise, and shall prove the 
following two theorems. 

\begin{thm} \label{l:TauPro}
For all $p \in (0, \alpha)$, as $T>T_0:=\frac{(p-1)\log 3}{p \lambda_1} \vee 0$, there
exist some $M, \tl \vartheta, C>0$ all depending on $p,\lambda, \|F\|_0, T$, so that for all $x, y \in \R^2$,
\Bes
\E_{(x,y)} [{\rm e}^{\tl \vartheta \tl \sigma(x,y,M)}]<C(1+|x|^p+|y|^p).
\Ees
\end{thm}

\begin{thm} \label{t:SigDEst}
There exists some constants $\vartheta, C>0$ depending on $p,  \lambda, \|F\|_{1}, d, K$, such that for all $p \in (0, \alpha)$ and $x, y \in \R^2$,
\Be \label{e:SigEst}
\E_{(x,y)} [{\rm e}^{\vartheta \sigma(x,y,d)}]\leq C\bg(1+|x|^p+|y|^p\bg).
\Ee
\end{thm}
\vskip 3mm

\begin{proof} [Proof of Theorem \ref{l:TauPro}]
It suffices to show that for every $p \in (0, \alpha)$, as $T>T_0:=\frac{(p-1)\log 3}{p \lambda_1} \vee 0$, there
exist some $M>0$ depending on $p,\lambda, \|F\|_0, T, \nu$
and some $q \in (0,1)$ depending on $p,\lambda,M$ such that
\begin{\eqn} \label{e:TauPro}
\PP_{(x,y)}(\tl \sigma>k) \leq q^k\left(1+|x|^p+|y|^p\right) \ \ \ \ \ \ k \ge 1,
\end{\eqn}
for all $x,y \in \R^2$. The proof of \eqref{e:TauPro} is by the same argument as that in Lemma 6.5 of \cite{PSXZ11}. To apply that argument,
we only need to show
\Be \label{e:SxSy<q}
\E\left(|S^x(1)|^p+|S^y(1)|^p\right) \le q^2(|x|^p+|y|^p)+C.
\Ee
where $C$ depends on $\lambda, p, \|F\|_0$.
\vskip 2mm

By Proposition \ref{p:CouChn}, for all $p \in (0,\alpha)$ we have
\Bes
\begin{split}
& \ \E \left(\left|S^{x}(1)\right|^p+\left|S^{y}(1)\right|^p \right)=\E \left|X^{x}(\tau_1)\right|^p+\E \left|X^{y}(\tau_1)\right|^p,
\end{split}
\Ees
which, together with Lemma \ref{l:SolEst}, implies
\Be
\E \left(\left|S^{x}(1)\right|^p+\left|S^{y}(1)\right|^p \right) \le
(3^{p-1} \vee 1) \E \left[\re^{-p \lambda_1 \tau_1}\right] (|x|^p+|y|^p)+C,
\Ee
where $C$ depends on $\lambda, \|F\|_{0}, p$. Therefore, to show \eqref{e:SxSy<q}, we only need to show that
\Be \label{e:C1EeSma}
(3^{p-1} \wedge 1) \E[\re^{-p \lambda_1 \tau_1}]<1.
\Ee
When $p \le 1$, \eqref{e:C1EeSma} automatically holds for all $T \ge 0$. When $p>1$,
\Bes
3^{p-1} \E[\re^{-p\lambda_1  \tau_1}]=\frac{3^{p-1} \gamma_K e^{-p \lambda_1  T}}{\gamma_K+p \lambda_1 }<1
\Ees
as $T>T_0$.
\end{proof}
\vskip 2mm

We are now at the position to show Theorem \ref{t:SigDEst}, to this end, we first show  
\begin{prop}   \label{p:SquKOrd}
For all $x, y \in \R^2$, we have
\Bes %\label{e:SquKOrd}
\begin{split}
\PP\left\{|S^x(k+1)-S^y(k+1)|>\delta_k |S^x(k)-S^y(k)| \big |S^{x,y}(k)\right\} \le \kappa |S^x(k)-S^y(k)|^{\beta_2},
\end{split}
\Ees
for all $k \ge 0$,
where
\Be \label{e:Kappa}
\delta_k={\rm e}^{-\lambda_2 (\tau_{k+1}-\tau_k)}+\frac{\|F\|_{{\rm Lip}} {\rm e}^{(\tau_{k+1}-\tau_k) \|F\|_{{\rm Lip}}}}{\lambda_2+\|F\|_{{\rm Lip}}},
\ \ \ \kappa=\beta_1 {\rm e}^{\beta_2 \|F\|_{{\rm Lip}} T},
\Ee
with $\beta_1, \beta_2$ being the constants in Assumption \ref{a:} and $T$ being defined in \eqref{e:TDef}. Furthermore, if $|S^x(k)|+|S^y(k)| \le M$,
\Bes
\PP\left\{|S^x(k+1)-S^y(k+1)|>\delta_k |S^x(k)-S^y(k)| \big |S^{x,y}(k)\right\} \le {\beta_0}/{2},
\Ees
with $\beta_0$ being the constant in Assumption \ref{a:}
\end{prop}

\begin{proof}
The proofs of the both inequalities are similar, we only show the first one.
Since $\{S^{x,y}(k)\}_{k \ge 0}$ is a time-homogeneous Markov chain,
it suffices to show the inequality for $k=0$, i.e.
\Be \label{e:Squ1Ord}
\PP\left(|S^x(1)-S^y(1)| \ge \delta_0 |x-y|\right) \le \kappa |x-y|^{\beta_2}.
\Ee
By the construction of the Markov chain $\{S^{x,y}(k)\}_{k \ge 0}$, $S^{x,y}(1)$ has the same distribution
as
\Be \label{e:XiCouX}
\big(\bar \xi_x(X^x(\tau_1-), X^y(\tau_1-)), \ \bar \xi_y(X^x(\tau_1-), X^y(\tau_1-))\big),
\Ee
we shall write
\eqref{e:XiCouX} by $(\bar \xi_x, \bar \xi_y)$ in shorthand.
By \eqref{e:BarXi}, we have
$
\bar \xi_x=\left[\begin{array}{lll} &\ \ \xi_x \\
&X^x_2(\tau_1-)
\end{array}
\right], \ \ \ \bar \xi_y=\left[\begin{array}{lll} & \ \ \xi_y \\
&X^y_2(\tau_1-)
\end{array}
\right],
$
thus
\Be \label{e:BarxNeqBary}
\begin{split}
\PP \left(|\bar \xi_x-\bar \xi_y|>\delta_0 |x-y|\right) & \le \PP \left(|\xi_x-\xi_y|+|X^x_2(\tau_1-)-X^y_2(\tau_1-)|>\delta_0 |x-y|\right) \\
& \le \PP \left(\xi_x \neq \xi_y\right)+\PP \left(\xi_x=\xi_y, \ |X^x_2(\tau_1-)-X^y_2(\tau_1-)|>\delta_0 |x-y|\right).
\end{split}
\Ee

\noindent It follows from Lemma \ref{l:TV1} that
\Bes
\begin{split}
 \PP \left(\xi_x \neq \xi_y\right)=\E \left[\PP\big(\xi_x \neq \xi_y \big |(X^x_1(\tau-),X^y_1(\tau_1-))\big)\right] \le  \beta_1\E \left|X^x_1(\tau_1-)-X^y_1(\tau_1-)\right|^{\beta_2},
\end{split}
\Ees
this, together with Lemma \ref{l:SolEst} and (A4) of Assumption \ref{a:}, implies
\Be
\begin{split}
\PP \left(\xi_x \neq \xi_y\right)  &\le \beta_1 \E\left[\re^{\beta_2 \|F\|_{{\rm Lip}} \tau_1}\right] |x-y|^{\beta_2}/2\le  \kappa |x-y|^{\beta_2}.
\end{split}
\Ee
From Lemma \ref{l:SolEst} we have
$$|X^x_2(\tau_1-)-X^y_2(\tau_1-)| \le \delta_0 |x-y| \ \ \ a.s., $$
thus,
\Be \label{e:Xx=Xy}
\PP \left(\ |X^x_2(\tau_1-)-X^y_2(\tau_1-)|>\delta_0 |x-y|\right)=0.
\Ee
Collecting \eqref{e:BarxNeqBary}-\eqref{e:Xx=Xy}, we immediately get the desired inequality.
\end{proof}
To prove Theorem \ref{t:SigDEst}, we also need 
\begin{lem} \label{l:Sx-Sy<12}
Let $x,y \in \R^2$ be such that $|x|+|y| \le M$.
As $\lambda_2$ is sufficiently large, depending on $\|F\|_1,K,T, \beta, M, d$, 
we have
\Bes %\label{e:Sx-Sy<12}
\PP\left\{|S^x(1)-S^y(1)|>d \right\}<(2+\beta_0)/ 4
\Ees
where $\beta_0 \in (0,2)$ is the constant defined in Assumption \ref{a:}.
\end{lem}
\begin{proof}
It is easy to have
\Bes
\begin{split}
\PP\left\{|S^{x}(1)-S^{y}(1)|>d\right\} &\le \PP\left\{|S^{x}(1)-S^{y}(1)|>\delta_0 |x-y|, \ \delta_0 \le d/M\right\}+\PP\left\{\delta_0>d/M\right\}
\end{split}
\Ees
where $\delta_0$ and $\kappa$ are defined in Proposition \ref{p:SquKOrd}. This inequality,
together with Proposition \ref{p:SquKOrd}, Markov inequality, implies
\Bes
\begin{split}
\PP\left\{|S^{x}(1)-S^{y}(1)|>d\right\} &\le \beta_0/2+M\E_{(x,y)} [\delta_0]/d\\
& \le \beta_0/2+M/d \left[\re^{-\lambda_2 T} \frac{\gamma_K}{\gamma_K+\lambda_2}+\frac{ 2\re^{\|F\|_{{\rm Lip}} T} \|F\|_{{\rm Lip}}}{\lambda_2+\|F\|_{{\rm Lip}}}\right].
\end{split}
\Ees
As $\lambda_2$ is sufficiently large, we get the desired inequality.
\end{proof}
\vskip 3mm

\begin{proof} [Proof of Theorem \ref{t:SigDEst}]
It suffices to show that
for every $p \in (0,\alpha)$, there exist some $\gamma>0$ and $C>0$ depending on
$p, \lambda, \|F\|_1, K$, so that
\Be \label{e:Taud>K}
\PP_{(x,y)} \left\{\sigma=k\right\} \le C\re^{-\gamma k} (1+|x|^p+|y|^p), \ \ \forall k>0, \ \ \forall x,y \in \R^2;
\Ee
and
\Be \label{e:TaudInf}
\PP_{(x,y)} \left\{\sigma=\infty\right\}=0.
\Ee
\vskip 2mm

\emph{Step
1}. Write $\tl \sigma _0=0$, define
\begin{\eqn*}
\tl \sigma_{k}=\inf\{j> \tl \sigma_{k-1}+1; |S^x(j)|+|S^y(j)| \leq M \}, \ \ \ k \in \N.
\end{\eqn*}
Since $(S(k))_{k \ge 0}$ is a discrete time
Markov chain, it is strong Markovian. Therefore,
it follows from Theorem \ref{l:TauPro} that
\begin{\eqn} \label{e:TauKK-1}
\E_{S(\tl \sigma_k)}\left[\re^{\tl \vartheta (\tl \sigma_{k+1}-\tl \sigma_k-1)}\right] \leq C (1+|S(\tl \sigma_k)|^{p}) \leq
C(1+M^{p}).
\end{\eqn}
where $C, \tl \vartheta$ depends on $\lambda, p, M, T, \|F\|_0$. The above inequality, together with strong Markov
property, implies
\begin{\eqn} \label{e:TauKEst}
\begin{split}
\E_{(x,y)}[\re^{\tl \vartheta \tl \sigma_k}] &=\E_{(x,y)} \left[\re^{\tl \vartheta \tl \sigma _1}\E_{S(\tl \sigma _1)}
\left[\re^{\tl \vartheta (\tl \sigma _2-\tl \sigma _1)}\cdots \E_{S(\tl \sigma _{k-1})}\left[\re^{\tl \vartheta (\tl \sigma _k-\tl \sigma _{k-1})}\right]\cdots \right]\right]\\
& \leq C^k \re^{\tl \vartheta k} (1+M^{p})^{k-1} (1+|x|^{p}+|y|^p).
\end{split}
\end{\eqn}
\vskip 3mm

 \emph{Step 2}. Given any $k \in \N$, define
$$\tl \rho_k=\sup\{j; \ \tl \sigma_j \leq k\}.$$
Clearly, ${\tl \sigma} _{\tl \rho_k+1}>k$ if $\tl \rho_k<\infty$. We have
\begin{\eqn}
\begin{split}
\PP_{(x,y)}(\sigma=k)&=\sum_{j=0}^k \PP_{(x,y)}(\sigma=k, \tl \rho_k=j) \le \sum_{j=0}^l \PP_{(x,y)}(\tl \rho_k=j)+\sum_{j=l+1}^k
\PP_{(x,y)}(\sigma=k, \tl \rho_k=j),
\end{split}
\end{\eqn}
where $l<k$ is to be chosen later. We denote 
$$I_1=\sum_{j=0}^l \PP_{(x,y)}(\tl \rho_k=j),\ \ \ \ I_2=\sum_{j=l+1}^k
\PP_{(x,y)}(\sigma=k, \tl \rho_k=j).$$
\vskip 2mm

\emph{Step 3}. Let us now estimate $I_1$ and $I_2$. By Chebyshev inequality and strong Markov
property, we have
\begin{\eqn*}
\begin{split}
\PP_{(x,y)}(\tl \rho_k=j) &\leq \PP_{(x,y)}\left(\tl \sigma_j \leq k/2, \ \tl \rho_k=j \right)
+\PP_{(x,y)} \left(\tl \sigma_j>k/2\right)
 \\
&\leq \PP_{(x,y)} \left(\tl \sigma_j \leq k/2, \ \tl \sigma _{j+1}>k\right)+\PP_{(x,y)} \left(\tl \sigma_j>k/2\right)
\\
& \leq \E_{(x,y)} \left[\PP_{S(\tl \sigma _j)} \left(\tl \sigma _{j+1}-
\tl \sigma_j>k/2\right) \right]+\re^{-\tl \vartheta k/2} \E_{(x,y)}[\re^{\tl \vartheta \tl \sigma _j}].
\end{split}
\end{\eqn*}
By \eqref{e:TauKK-1} and \eqref{e:TauKEst}, the above inequality
implies
\begin{\eqn*}
\begin{split}
\PP_{(x,y)}(\tl \rho_k=j) \leq C \re^{\tl \vartheta}(1+M^{p}) \re^{-\tl \vartheta k/2}+C^j \re^{j \tl \vartheta} 
(1+M^{p})^{j-1} (1+|x|^{p}+|y|^p)\re^{-\tl \vartheta
k/2}.
\end{split}
\end{\eqn*}
Hence,
\begin{\eqn}
\begin{split}
 I_1  \leq  (C \re^{\tl \vartheta})^{l+2} (1+M^{p})^{l+2} (1+|x|^{p}+|y|^p) \re^{-\tl \vartheta k/2}.
\end{split}
\end{\eqn}

\noindent Next we estimate $I_2$. For $j \in \N$, define
$$A_j:=\left\{|S^x(\tl \sigma _1+1)-S^y(\tl \sigma _1+1)|>d, \ldots, |S^x(\tl \sigma _{j}+1)-S^y(\tl \sigma _{j}+1)|>d\right\}.$$
By the definitions of $\sigma$ and
$\tl \rho_k$, strong Markov property, we have
\begin{\eqn*}
\begin{split}
\PP_{(x,y)}\left(\sigma=k, \tl \rho_k=j\right) \le \PP_{(x,y)}(A_{j-1})&=\PP_{(x,y)}\left\{|S^x(\tl \sigma _{j-1}+1)-S^y(\tl \sigma _{j-1}+1)|>d, A_{j-2}\right\} \\
&=\E_{(x,y)}\left\{\PP_{u}\left\{|S^{u_x}(1)-S^{u_y}(1)|>d\right\} A_{j-2}\right\},
\end{split}
\end{\eqn*}
where $u=S^{x,y}(\tl \sigma_{j-1})$.
Combining with Lemma \ref{l:Sx-Sy<12}, the above inequality implies
\begin{\eqn*}
\PP_{(x,y)}\left(\sigma=k, \tl \rho_k=j\right)  \le  \left(\frac 12+\frac{\beta_0} 4\right) \PP_{(x,y)} (A_{j-2}) \le  \left(\frac 12+\frac{\beta_0} 4\right)^{j-1}
\end{\eqn*}
Hence,
\begin{\eqn}
I_2 \leq \left(\frac 12+\frac{\beta_0} 4\right)^l\bigg/\left(\frac 12-\frac{\beta_0} 4\right).
\end{\eqn}
Take $l=\e k$, it follows from the bounds of $I_1$ and $I_2$ that
as $\e>0$ is sufficiently small, \eqref{e:Taud>K} follows.
\vskip 1mm

\emph{Step 4:} Let us now show \eqref{e:TaudInf}. Define
\Bes
\tl \rho_\infty=\sup \{j; \tl \sigma_j<\infty\},
\Ees
it is clear that $\tl \sigma_{\tl \rho_\infty+1}=\infty$ if $\tl \rho_\infty<\infty$.
For all $j \in \N \cup \{0\}$, by strong Markov property and Theorem \ref{l:TauPro}
we have
\Be
\PP_{(x,y)} \left(\tl \rho_\infty=j\right)=\E_{(x,y)} \left[\PP_{S(\tl \sigma_j)}  \left(\tl \sigma_{j+1}-\tl \sigma_j=\infty\right)\right]=0.
\Ee
Hence, $$\PP_{(x,y)} \left(\tl \rho_\infty=\infty\right)=1 \ \ \ \forall \  x, y \in \R^2.$$
By a similar computation as estimating $I_2$ in step 3, we have
\Bes
\PP_{(x,y)} \left(\sigma=\infty\right)=\PP_{(x,y)} \left(\sigma=\infty, \tl \rho_\infty=\infty\right)
\le \PP_{(x,y)}(A_j) \le \left(\frac 12+\frac{\beta_0} 4\right)^j \rightarrow 0, \ \ j \rightarrow \infty.
\Ees
\end{proof}

\section{Proof of the main theorem} \label{s:PfMaiThm}
Define
\Be \label{d:SigXY}
\hat \sigma (x,y):=\inf\{k \ge 1; |S^x(k)-S^y(k)|>(\delta_0 \dots \delta_{k-1}) |x-y|\}
\Ee
where $\delta_j$ ($j=0, ..., k-1$) are defined in Proposition \ref{p:SquKOrd}, we shall often
write $\hat \sigma=\hat \sigma(x,y)$ in shorthand.

\begin{lem} \label{l:SigInf}
If $|x-y| \le d$ with $0<d<\left(\frac 1{4 \kappa}\right)^{1/\beta_2}$ and $\kappa$ defined in
Proposition \ref{p:SquKOrd}, as $\lambda_2>0$ is sufficiently large, depending on $T,K,\|F\|_{{\rm Lip}}, \beta$,
we have
$$\PP_{(x,y)} (\hat \sigma=\infty)>1/2.$$
Moreover, there exists some $\epsilon, C>0$ depending on
$d,\lambda, \|F\|_1, K$ such that
$$\E_{(x,y)} [\re^{\epsilon \hat \sigma} 1_{\{\hat \sigma<\infty\}}] \le C.$$
\end{lem}
\begin{proof}
For all $k \ge 0$, define
$$B_k:=\left\{|S^x(k+1)-S^y(k+1)|>\delta_k |S^x(k)-S^y(k)|\right\},$$
$$C_k:=\left\{|S^x(j+1)-S^y(j+1)| \le (\delta_0 \dots \delta_j) |x-y|, \ \ 0 \le j \le k\right\},$$
it is easy to see that $C_k \supset B^c_k \cap C_{k-1}.$
It follows from Proposition \ref{p:SquKOrd} that
\Bes
\begin{split}
\PP_{(x,y)}\left(C_k\right)& \ge \PP_{(x,y)}\left(B^c_k \cap C_{k-1}\right)=\E_{(x,y)} \left[\PP \left(B^c_k\big |S^{x,y}(k)\right) 1_{C_{k-1}}\right] \\
&=\E_{(x,y)} \left\{\left[1-\PP \left(B_k \big |S^{x,y}(k)\right)\right] 1_{C_{k-1}}\right\} \ge \E_{(x,y)} \left[\left(1-\kappa |S^x(k)-S^y(k)|^{\beta_2}\right) 1_{C_{k-1}}\right] \\
&\ge \PP_{(x,y)}(C_{k-1})-\kappa \E_{(x,y)}\left[(\delta_0 \dots \delta_{k-1})^{\beta_2}\right] |x-y|^{\beta_2}.
\end{split}
\Ees
This inequality, together with \eqref{e:TauDen}, \eqref{e:TauKInd}, (A4) of Assumption \ref{a:}, implies that
\Bes
\begin{split}
\PP_{(x,y)}\left(C_k\right) & \ge \PP_{(x,y)}(C_{k-1})-\kappa \theta^k |x-y|^{\beta_2} \\
& \ge \PP_{(x,y)}(C_{k-2})-\kappa \theta^{k-1} |x-y|^{\beta_2}-\kappa \theta^k |x-y|^{\beta_2} \\
& \ge 1-\kappa \frac{1-\theta^{k+1}}{1-\theta}|x-y|^{\beta_2}>1/2
\end{split}
\Ees
where 
$$\theta=\frac{\gamma_K}{\gamma_K+\lambda_2 \beta_2} {\rm e}^{-\lambda_2 \beta_2 T}+2\left(\frac{\|F\|_{{\rm Lip}}}{\|F\|_{{\rm Lip}}+\lambda_2}\right)^{\beta_2}<\frac 12, $$
as long as $\lambda_2>0$ is sufficiently large. Thus, we get the first inequality.

Defining $D_k:=\{|S^x(k+1)-S^y(k+1)|>(\delta_0 \dots \delta_k) |x-y|\}$ for all $k \ge 0$,
by a similar calculation as above we have
\Bes
\begin{split}
\PP_{(x,y)}(\hat \sigma=k)&=\PP_{(x,y)}\big(D_{k-1} \cap C_{k-2}\big) \le \PP_{(x,y)}\big(B_{k-1} \cap C_{k-2}\big) \\
&=\E_{(x,y)}\left[\PP\left(B_{k-1}\big | S^{x,y}(k-1)\right) 1_{C_{k-2}}\right]\le \kappa \E_{(x,y)}\big[(\delta_0 ... \delta_{k-2})^{\beta_2}\big] |x-y|^{\beta_2} \\
&\le \frac{1}{2} \theta^{k-1}  \le (\frac12)^k.
\end{split}
\Ees
This immediately implies the second inequality.
\end{proof}
\vskip 2mm

Define
\Be \label{d:SigDag}
\sigma^{\dag}(x,y,d):=\sigma+\hat \sigma(S^{x,y}(\sigma))
\Ee
where $\sigma=\sigma(x,y,d)$ is defined by \eqref{d:SigD}. Further define
\Be \label{d:BarSig}
\bar \sigma(x,y,d,M):=\sigma^\dag+\tl \sigma(S^{x,y}(\sigma^\dag),M).
\Ee
where $\sigma^\dag=\sigma^\dag(x,y,d)$ and $\tl \sigma$ are defined in \eqref{d:SigM}.

The motivation for defining $\bar \sigma$ is the following: we only know
$|S^x(\sigma^\dag)-S^y(\sigma^\dag)| \le d$, but have no idea about
the bound of $|S^x(\sigma^\dag)|+|S^y(\sigma^\dag)|$.
This bound is very important for iterating a stopping time argument as in Step 1 of
the proof of Theorem \ref{t:SigDEst}. To this aim, we introduce \eqref{d:BarSig}
and thus have
\Be \label{e:SBarSigM}
|S^{x,y}(\bar \sigma)| \le M \ \ \ \ \ \forall x,y \in \R^2.
\Ee

\begin{lem} \label{l:BarSigEst}
Let $0<d<\left(\frac 1{4 \kappa}\right)^{1/\beta_2}$ and $p \in (0, \alpha)$. There exist some $\gamma, C>0$
depending on $d,\lambda, \|F\|_1, p, M, K$ such that
\Bes
\E_{(x,y)}\left[\re^{\gamma \bar \sigma (x,y,d,M)} 1_{\{\bar \sigma(x,y,d,M)<\infty\}}\right] \le C(1+|x|^p+|y|^p).
\Ees
\end{lem}

\begin{proof}
Note that $\sigma<\infty$ a.s. by Theorem \ref{t:SigDEst}.
By the strong Markov property we have
\Bes
\E_{(x,y)} \left[\re^{\gamma \sigma^{\dag} (x,y,d)} 1_{\{\sigma^{\dag}(x,y,d)<\infty\}}\right]=\E_{(x,y)}\left\{\E_u[\re^{\gamma \hat \sigma} 1_{\{\hat \sigma<\infty\}}] \ e^{\gamma \sigma} \right\}
\Ees
where $\sigma=\sigma(x,y,d)$, $u=S^{x,y}(\sigma)$, $\hat \sigma=\hat \sigma(S^{x,y}(\sigma))$.
\vskip 2mm
By (2) of Lemma \ref{l:SigInf} and Theorem \ref{t:SigDEst}, as $\gamma>0$ is sufficiently small we immediately get
\Be \label{e:DagSigEst}
\E_{(x,y)} \left[\re^{\gamma \sigma^{\dag} (x,y,d)} 1_{\{\sigma^{\dag}(x,y,d)<\infty\}}\right]  \le C(1+|x|^p+|y|^p)
\Ee
where $C$ is some constant depending on $d,\lambda, \|F\|_1, p, K$.

\vskip 2mm

By strong Markov property and the above inequality, we have
\Be \label{e:BarSigSM}
\begin{split}
\E_{(x,y)}\left[\re^{\gamma \bar \sigma (x,y,d,M)} 1_{\{\bar \sigma(x,y,d,M)<\infty\}}\right] &\le \E_{(x,y)}\left\{\E_u \left[\re^{\gamma \tl \sigma (u,M)}\right] \re^{\gamma \sigma^\dag (x,y,d)} 1_{\{\sigma^\dag(x,y,d)<\infty\}} \right\}  \\
& \le C \E_{(x,y)}\left[(1+|u|^{p/2}) \re^{\gamma \sigma^\dag (x,y,d)} 1_{\{\sigma^\dag(x,y,d)<\infty\}}\right]
\end{split}
\Ee
where $u=S^{x,y}(\sigma^\dag)$ and $C$ depends on $M, \lambda, \|F\|_1,  p, K$.
Note from Lemma \ref{l:SolEst} that
$$\E |S^{x,y}(\sigma^\dag)|^{p}=\E |X^{x}(\tau_{\sigma^\dag})|^{p}+\E |X^{y}(\tau_{\sigma^\dag})|^{p} \le C(1+|x|^p+|y|^{p}).$$
The inequality \eqref{e:BarSigSM}, together with H\"{o}lder inequality and \eqref{e:DagSigEst}, immediately implies the desired inequality
as $\gamma>0$ is sufficiently small.
\end{proof}
\vskip 2mm

Define $\bar \sigma_0=0$, for all $k \ge 0$ we define
$$\bar \sigma_{k+1}=\bar \sigma_k+\bar \sigma(S^{x,y}(\bar \sigma_k), d, M).$$

\begin{lem} \label{l:BarSigK}
For all $x,y \in \R^2$, we have
\Be
\PP_{(x,y)} \left(\bar \sigma_k<\infty\right) \le 1/2^k, \ \ k \in \N.
\Ee
\end{lem}

\begin{proof}
It follows from Theorem \ref{t:SigDEst} that $\sigma<\infty$ a.s..
By the definition of $\bar \sigma$, strong Markov property, Lemma \ref{l:SigInf}, for all $x,y \in \R^2$ we have
$$\PP_{(x,y)} (\bar \sigma=\infty) = \E_{(x,y)} \left[\PP_{S^{x,y}(\sigma)}\left(\hat \sigma=\infty\right)\right]>1/2.$$
This, together with strong Markov property, implies that as $\bar \sigma_{k-1}<\infty$,
\Bes
\PP_{u} (\bar \sigma_{k}-\bar \sigma_{k-1}=\infty)>1/2,
\Ees
where $u=S^{x,y}(\bar \sigma_{k-1})$. Hence,
\Bes
\begin{split}
& \ \ \ \PP_{(x,y)} \left(\bar \sigma_k<\infty\right)=\PP_{(x,y)} \left(\bar \sigma_k<\infty, \bar \sigma_{k-1}<\infty\right) \\
& \le \E_{(x,y)} \left[\PP_u \left(\bar \sigma_k-\bar \sigma_{k-1}<\infty\right) 1_{\{\bar \sigma_{k-1}<\infty\}}\right] \le \frac 12 \PP_{(x,y)} \big(\bar \sigma_{k-1}<\infty \big) \\
& \le \frac 14 \PP_{(x,y)} \big(\bar \sigma_{k-2}<\infty\big) \le ... \le \frac 1{2^k}.
\end{split}
\Ees
\end{proof}

\begin{proof} [Proof of Theorem \ref{t:MaiThm}]
The existence of invariant measures has been established in \cite{PXZ10}.
According to Section 2.2. of \cite{Shi08}, the inequality \eqref{2.5} in the theorem implies
the uniqueness of the invariant measure. So now we only need to show \eqref{2.5}, by
\cite{Shi08} again,
it suffices to show that for all $p \in (0, \alpha)$ we have
\Be \label{e:EquInqMai}
 |P_tf(x)-P_t f(y)| \leq C \re^{-ct}\|f\|_1 (1+|x|^{p}+|y|^p) \ \ \ \forall \ f \in L_b(\R^2),
\Ee
where $C,c$ depend on $p,\beta,K, \|F\|_{1}, \lambda$. Let us prove \eqref{e:EquInqMai} by the following four steps.

\vskip 2mm

\emph{Step 1}. Let $l \ge 2$ be some natural number to be determined later. We easily have
\Bes
\left|\E[f(X^x(t))]-\E[f(X^y(t))]\right| \le I_1+I_2
\Ees
with
\Bes
\begin{split}
& I_1:=\left|\E \left\{\left[f(X^x(t))-f(X^y(t))\right]1_{\left\{\tau_{l-1}>t\right\}} \right\}\right|, \\
& I_2:=\left|\E \left\{\left[f(X^x(t))-f(X^y(t))\right]1_{\left\{\tau_{l-1} \le t\right\}} \right\}\right|.
\end{split}
\Ees
By \eqref{e:TauDen}, we have
$\E \re^{\tau_j \gamma_K/2}=\re^{\gamma_K T j/2} 2^j$ for all $j \in \N$, thus
\Be \label{e:EstI1Thm}
I_1 \le 2 \|f\|_0 \PP(\tau_{l-1}>t) \le 2 \|f\|_0 \left(2 \re^{\gamma_K T/2}\right)^{l-1} \re^{-\gamma_K t/2}.
\Ee
\vskip 3mm

\emph{Step 2.} Now we estimate $I_2$. Observe
\Be
\begin{split}
I_2 \le I_{2,1}+I_{2,2},
\end{split}
\Ee
where
$$I_{2,1}=\left|\E \left\{\left[f(X^x(t))-f(X^y(t))\right]1_{\left\{\tau_{l-1} \le t<\tau_{l}\right\}} \right\}\right|,$$
$$I_{2,2}=\left|\E \left\{\left[f(X^x(t))-f(X^y(t))\right]1_{\left\{\tau_{l} \le t\right\}} \right\}\right|.$$
By a similar argument as for $I_1$, we have
\Bes
I_{2,1} \le 2 \|f\|_0 \PP(\tau_l>t) \le  2 \|f\|_0 \left(2 \re^{\gamma_K T/2}\right)^{l} \re^{-\gamma_K t/2}.
\Ees
\vskip 3mm
For $I_{2,2}$, by strong Markov property we have
\Bes
I_{2,2}=\left|\E \left[ \left(g(X^x(\tau_l))-g(X^y(\tau_l))\right)1_{\left\{\tau_{l} \le t\right\}}\right]\right|
\Ees
where
$$g(X^x(\tau_l))=\E[f(X^x(t))|X^x(\tau_l)], \ \ \ \ \ g(X^y(\tau_l))=\E[f(X^y(t))|X^y(\tau_l)].$$
By strong Markov property again, on the set $\{\tau_l \le t\}$ we have
$$g(u_x)=\E\left[f(X^{u_x}(t-\tau_l))\right], \ \ \ \ g(u_y)=\E\left[f(X^{u_y}(t-\tau_l))\right],$$
where $u_x=X^x(\tau_l), u_y=X^y(\tau_l)$,
by the third inequality in Lemma \ref{l:SolEst} we further have
\Be \label{e:LipG}
\begin{split}
|g(u_x)-g(u_y)| &\le \E \left[\left|f(X^{u_x}(t-\tau_l))-f(X^{u_y}(t-\tau_l))\right|\right] \\
& \le \|f\|_1 \E \left[\left|X^{u_x}(t-\tau_l)-X^{u_y}(t-\tau_l)\right|\right]  \\
& \le \|f\|_1 \re^{t \|F\|_{\rm Lip}}|u_x-u_y|.
\end{split}
\Ee
By Proposition \ref{p:CouChn} and the easy fact $\|g\|_0 \le \|f\|_0$, we have
\Bes
\begin{split}
I_{2,2} & \le \left|\E \left[g(X^x(\tau_l))-g(X^y(\tau_l))\right]\right|+\left|\E\left\{\left[g(X^x(\tau_l))-g(X^y(\tau_l))\right] 1_{\left\{\tau_l>t\right\}} \right\}\right| \\
& \le \left|\E \left[g(S^x(l))-g(S^y(l))\right]\right|+2\|f\|_0 \PP(\tau_l>t) \\
& \le \left|\E \left[g(S^x(l))-g(S^y(l))\right]\right|+2 \|f\|_0 \left(2 \re^{\gamma_K T/2}\right)^l \re^{-\gamma_K t/2}.
\end{split}
\Ees
where the last inequality is by a similar calculation as for $I_1$.

\vskip 2mm

\emph{Step 3.}
Let $m=[\e l]$ with $0<\e<1/2$ to be determined later. We have
$$\left|\E \left[g(S^x(l))-g(S^y(l))\right]\right|=J_{1}+J_{2},$$
where
\Bes
\begin{split}
 J_{1}=\left|\E \left\{\left[g(S^x(l))-g(S^y(l))\right]1_{\{\bar \sigma_m<\infty\}}\right\}\right|,\ \ \  J_{2}=\left|\E \left\{\left[g(S^x(l))-g(S^y(l))\right]1_{\{\bar \sigma_m=\infty\}}\right\}\right|.
\end{split}
\Ees
By the easy fact $\|g\|_0 \le \|f\|_0$ and Lemma \ref{l:BarSigK}, we have
\Be
J_{1} \le 2\|f\|_0 \PP_{(x,y)}\{\bar \sigma_m<\infty\} \le \frac {\|f\|_0}{2^{m-1}} \le \|f\|_0 2^{-\e l+1}.
\Ee
Observe
\Bes
%\begin{split}
\E \left\{\left|g(S^x(l))-g(S^y(l))\right|1_{\{\bar \sigma_m=\infty\}}\right\}=J_{2,1}+J_{2,2}.
\Ees
where
\Bes
\begin{split}
&J_{2,1}=\sum_{i=0}^{m-1}\E \left\{\left|g(S^x(l))-g(S^y(l))\right|1_{\{l/2<\bar \sigma_i<\infty, \ \bar \sigma_{i+1}=\infty\}}\right\}, \\
&J_{2,2}=\sum_{i=0}^{m-1}\E \left\{\left|g(S^x(l))-g(S^y(l))\right|1_{\{\bar \sigma_i \le l/2, \ \bar \sigma_{i+1}=\infty\}}\right\}.
\end{split}
\Ees
By Chebyshev inequality and Lemma \ref{l:BarSigEst}
\Bes
\begin{split}
J_{2,1}  & \le 2\|f\|_0 \sum_{i=0}^{m-1}\PP_{(x,y)}\left\{\frac l2<\bar \sigma_i<\infty\right\} 
\le 2\|f\|_0 \re^{-\frac l 2 \gamma} \sum_{i=0}^{m-1}   \E_{(x,y)}[\re^{\gamma \bar \sigma_i}],
\end{split}
\Ees
and 
\Bes
\begin{split}
\E_{(x,y)}[\re^{\gamma \bar \sigma_i}]&=\E_{(x,y)} \left[\re^{\gamma \bar \sigma_1}\E_{S(\bar \sigma_1)}
\left[\re^{\gamma (\bar \sigma_2-\bar \sigma_1)}\cdots \E_{S(\bar \sigma_{i-1})}\left[\re^{\gamma (\bar \sigma_i-\bar \sigma_{i-1})}\right]\cdots \right]\right]\\
& \leq C^i \re^{\gamma i} (1+M^{p})^{i-1} (1+|x|^{p}+|y|^p), 
\end{split}
\Ees
where the last inequality is by \eqref{e:SBarSigM}.
Hence, 
$$J_{2,1} \le 2\|f\|_0 \re^{-\frac l 2 \gamma} \sum_{i=0}^{m-1} \re^{\gamma i} (1+M^{p})^{i-1} (1+|x|^{p}+|y|^p)$$
Recall $m=[\e l]$, as $\e>0$ is small enough we have
\Bes
J_{2,1} \le \re^{-\frac{\gamma l}4} \|f\|_0 (1+|x|^p+|y|^p).
\Ees
\vskip 2mm
It remains to estimate $J_{2,2}$. Recall the definition of
$\sigma, \hat \sigma, \sigma^\dag, \bar \sigma, \tl \sigma$ and note that
\Be \label{e:BarSigMore}
\bar \sigma_{i+1}=\bar \sigma_i+\sigma+\hat \sigma+\tl \sigma,
\Ee
with
$\sigma=\sigma(S^{x,y}(\bar \sigma_i),d)$,
$\hat \sigma=\hat \sigma(S^{x,y}(\bar \sigma_i+\sigma))$, $\tl \sigma=\tl \sigma(S^{x,y}(\bar \sigma_i+\sigma+\hat \sigma),M)$.
Observe that
\Bes
J_{2,2}=J_{2,2,1}+J_{2,2,2},
\Ees
with
\Bes
\begin{split}
& J_{2,2,1}:=\sum_{i=0}^{m-1}\E \left[\left|g(S^x(l))-g(S^y(l))\right|1_{\{\bar \sigma_i \le l/2, \ \bar \sigma_i+\sigma>\frac{3l}4, \bar \sigma_{i+1}=\infty\}}\right], \\
&J_{2,2,2}:=\sum_{i=0}^{m-1} \E \left[\left|g(S^x(l))-g(S^y(l))\right|1_{\{\bar \sigma_i \le l/2, \bar \sigma_i+\sigma \le \frac{3l}4, \bar \sigma_{i+1}=\infty\}}\right].
\end{split}
\Ees
By strong Markov property, Chebyshev inequality, Theorem \ref{t:SigDEst} and the clear fact $|S^{x,y}(\bar \sigma_i)|\le M$ for all $i \ge 1$, as $\e>0$
is sufficiently small we have
\Be
\begin{split}
J_{2,2,1} & \le 2 \|f\|_0 \sum_{i=0}^{m-1} \E_{(x,y)}  \left[\PP_{u_i}(\sigma>l/4) \right] \\
& \le  C  \|f\|_0 \re^{-\vartheta l/4}\big[(m-1)(1+M^p)+(1+|x|^p+|y|^p)\big] \\
& \le C\|f\|_0 \re^{-\vartheta l/8} (1+|x|^p+|y|^p)
\end{split}
\Ee
where $u_i=S^{x,y}(\bar \sigma_i)$ and $C, \vartheta$ depend on  $d,\lambda,\|F\|_1, p, M$.
\vskip 2mm

As for $J_{2,2,2}$, recall \eqref{e:BarSigMore} and note $\tl \sigma<\infty$ a.s. from Theorem \ref{l:TauPro}, we have
\Be
\begin{split}
J_{2,2,2}
&=\sum_{i=0}^{m-1} \E \left\{\left|g(S^x(l))-g(S^y(l))\right|1_{\{\bar \sigma_i \le l/2, \bar \sigma_i+\sigma \le \frac{3l}4, \hat \sigma+\tl \sigma=\infty\}}\right\} \\
&=\sum_{i=0}^{m-1} \E \left\{\left|g(S^x(l))-g(S^y(l))\right|1_{\{\bar \sigma_i \le l/2, \bar \sigma_i+\sigma \le \frac{3l}4, \hat \sigma=\infty\}}\right\}.
\end{split}
\Ee
 It follows from the above equality, \eqref{e:LipG} and strong Markov property that
\Bes
\begin{split}
J_{2,2,2} & \le \|f\|_1 \re^{t \|F\|_{\rm Lip}} \sum_{i=0}^{m-1} \E\left[|S^x(l)-S^y(l)|
1_{\{\bar \sigma_i+\sigma \le \frac{3l}4, \hat \sigma=\infty\}}\right] \\
& \le \|f\|_1 \re^{t \|F\|_{\rm Lip}} \sum_{i=0}^{m-1} \E\left[|S^x(l)-S^y(l)| 1_{\{\bar \sigma_i+\sigma \le \frac{3l}4, \hat \sigma=\infty\}}\right]
\\
&=\|f\|_1 \re^{t \|F\|_{\rm Lip}} \sum_{i=0}^{m-1} \E\left[\E_{u}\left(|S^x(l)-S^y(l)| 1_{\{\hat \sigma=\infty\}}\right)
1_{\{\bar \sigma_i+\sigma \le \frac{3l}4\}}\right] \\
\end{split}
\Ees
where $u=S^{x,y}(\bar \sigma_i+\sigma)$. By the definition of $\sigma$ we have $|u_x-u_y|<d$. By the definition \eqref{d:SigXY} with $\hat \sigma=\hat \sigma(S^{x,y}(\bar \sigma_i+\sigma))$ and the previous inequality, as $\lambda_2>0$ is sufficiently large, depending on $T,K,\|F\|_{\rm Lip}$, we have
\Bes
\begin{split}
J_{2,2,2}  & \le \|f\|_1  \re^{t \|F\|_{\rm Lip}} \sum_{i=0}^{m-1} \E\left[\E_u(\delta_0 ... \delta_{l/4})|u_x-u_y|\right] \\
& \le  d m\|f\|_1 \re^{t \|F\|_{\rm Lip}}\bigg(\frac{\gamma_K}{\gamma_K+\lambda_2}e^{-\lambda_2 T}+2 \frac{\|F\|_{Lip}}{\|F\|_{Lip}+\lambda_2}\bigg)^{l/4} \\
& \le \e d l\|f\|_1  \re^{t \|F\|_{\rm Lip}}\bigg(\frac 2{\lambda_2}\bigg)^{l/4}.
\end{split}
\Ees

\vskip 2mm

\emph{Step 4}.
Collecting the bounds for $J_{2,2,1}, J_{2,2,2}$, $J_{2,1}$, $J_{1}$, we have that there exist some $\epsilon, C>0$ depending on
$p, \lambda,\|F\|_{1}, K$ such that
\Bes
\begin{split}
I_{2,2} & \le J_1+J_2+2 \|f\|_0 \left(2 e^{\gamma_K T/2}\right)^l \re^{-\gamma_K t/2} \\
& \le J_1+J_{2,1}+J_{2,2,1}+J_{2,2,2}+2 \|f\|_0 \left(2 e^{\gamma_K T/2}\right)^l \re^{-\gamma_K t/2} \\
& \le \|f\|_0 2^{-\e l+1}+C\|f\|_0 \re^{-\vartheta l/8} (1+|x|^p+|y|^p)+\e d l\|f\|_1  \re^{t \|F\|_{\rm Lip}}\bigg(\frac 2{\lambda_2}\bigg)^{l/4}
+2 \|f\|_0 \left(2 \re^{\gamma_K T/2}\right)^l \re^{-\gamma_K t/2}
\end{split}
\Ees
Choosing $l=[\delta t]$ with $\delta>0$ sufficiently small (depending on $p, \lambda,\|F\|_{1}, K, \beta, M$) and then choosing $\lambda_2$ sufficiently large (depending on $p, \lambda,\|F\|_{1}, K, \beta, M, \delta$), we immediately get
$$I_{2,2} \le C_0 \re^{-c_0 t} \|f\|_1(1+|x|^p+|y|^p),$$
where $C_0$ and $c_0$ both depending on $p, \lambda,\|F\|_{1}, K, \beta, M$. 

Combining the estimates of $I_1, I_{2,1}$ and $I_{2,2}$, we immediately get the desired \eqref{e:EquInqMai}.
\end{proof}

%%%%%%%%%%%%%%%%%%%%%%%%%%%%%%%%%%%%%%%%%%

\bibliographystyle{amsplain}

\end{document}